%% file: 110206ExcSurg-2pp_arXiv.tex
\documentclass{amsart}

\usepackage{amssymb}
\usepackage{graphicx}
\usepackage{color}
\usepackage[all]{xy}

\newtheorem{theorem}{Theorem}[section]

\newtheorem{lemma}[theorem]{Lemma}
\newtheorem{corollary}[theorem]{Corollary}
\newtheorem{claim}{Claim}

\theoremstyle{definition}
\newtheorem*{ack}{Acknowledgments}

\theoremstyle{remark}

\newcommand{\sign}{\mathop{\mathrm{sign}}\nolimits}

\begin{document}

\title{Exceptional surgeries on $(-2,p,p)$-pretzel knots}

\date{\today}

\author{Kazuhiro Ichihara}
\address{Department of Mathematics, College of Humanities and Sciences, Nihon University,
3-25-40 Sakurajosui, Setagaya-ku, Tokyo 156-8550, Japan}
\email{ichihara@math.chs.nihon-u.ac.jp}

\author{In Dae Jong}
\address{Osaka City University Advanced Mathematical Institute, 3-3-138, Sugimoto, Sumiyoshi-ku, Osaka 558-8585, Japan}
\email{jong@sci.osaka-cu.ac.jp}

\author{Yuichi Kabaya}
\address{Osaka City University Advanced Mathematical Institute, 3-3-138, Sugimoto, Sumiyoshi-ku, Osaka 558-8585, Japan}
\email{kabaya@sci.osaka-cu.ac.jp}

\begin{abstract}
We give a complete description of exceptional surgeries on pretzel knots of type $(-2, p, p)$ with $p \ge 5$. 
It is known that such a knot admits a unique toroidal surgery yielding a toroidal manifold with a unique incompressible torus. 
By cutting along the torus, we obtain two connected components, one of which is a twisted $I$-bundle over the Klein bottle. 
We show that the other is homeomorphic to the one obtained by certain Dehn filling on the magic manifold. 
On the other hand, we show that all such pretzel knots admit no Seifert fibered surgeries. 
\end{abstract}



\maketitle

\section{Introduction}

By the Hyperbolic Dehn Surgery Theorem \cite[Theorem 5.8.2]{Thurston1978}, 
all Dehn surgeries on a hyperbolic knot give hyperbolic manifolds 
with only finitely many exceptions. 
Thus a Dehn surgery on a hyperbolic knot 
creating a non-hyperbolic manifold 
is now called an \textit{exceptional surgery}, 
on which a large mount of studies have been done. 
See \cite{Boyer2002} for a survey on this topic for example.

In this paper, we consider exceptional surgeries on 
pretzel knots $P(-2,p,q)$ in the 3-sphere $S^3$ 
of type $(-2,p,q)$ with $p,q \ge 3$. 
Here note that $p$ and $q$ must be odd 
otherwise $P(-2,p,q)$ has two or more components. 
Also we assume that $(p,q) \ne (3,3), (3,5)$ 
since $P(-2,3,3) = T(3,4)$ and $P(-2,3,5) = T(3,5)$ are non-hyperbolic, where $T(x,y)$ denote a torus knot of type $(x,y)$. 
All the other knots $P(-2,p,q)$ are known to be hyperbolic. 
See \cite{Oertel1984}, \cite{Menasco1984}, and 
\cite{BonahonSiebenmann1979-85, BonahonSiebenmann2010}.

Exceptional surgeries on such pretzel knots have been studied extensively, 
motivated by the fact that the class of the knots includes 
various interesting examples about exceptional surgeries. 
See \cite{FuterIshikawaKabayaMattmanShimokawa2009}, \cite{IchiharaJong2009,IchiharaJong2010a,IchiharaJong2011}, and \cite{Wu1996,Wu2006,Wu2009,Wu2010} for example.

We here recall that 
exceptional surgeries are classified into the following three types: 
a reducible surgery (yielding a reducible manifold), 
a toroidal surgery (yielding a toroidal 3-manifold), 
a Seifert fibered surgery (yielding a Seifert fibered 3-manifold), 
which is a consequence of an affirmative answer 
to the Geometrization Conjecture. 

\medskip

Our first result concerns toroidal surgeries on the knots. 
It is known that only the $2(p+q)$-surgery on 
the $(-2,p,q)$-pretzel knot is toroidal, 
and the surgered manifold contains 
the unique embedded incompressible torus up to isotopy. 
See \cite{Wu2010} for detailed descriptions. 

\begin{theorem}\label{thm:toroidal_kabaya}
Consider the toroidal manifold obtained by the $2(p+q)$-surgery on 
the hyperbolic $(-2,p,q)$-pretzel knot with odd integers $3 \le p \le q$. 
Let $M_{p,q}$ be the one of the two components obtained from the toroidal manifold 
by cutting along the unique embedded incompressible torus, 
which is not a twisted $I$-bundle over the Klein bottle. 
Then $M_{p,q}$ is homeomorphic to 
the manifold obtained by a $(-(k+1)/k,-(l+1)/l)$-surgery on the chain-link with three components, where $p = 2k + 1$ and $q=2l+1$. 
\end{theorem}

The chain-link with three components is depicted in 
Figure~\ref{fig:chain_link}, 
whose complement is called the ``magic manifold''. 
For its definition and notation, 
see \cite{martelli-petronio} in detail. 

In particular, in \cite{martelli-petronio}, 
the exceptional surgeries on the link are 
completely determined and classified. 
By referring their classification, we immediately obtain the following:

\begin{corollary}\label{cor:(-2,p,p)Tor}
Under the same setting as in Theorem~\ref{thm:toroidal_kabaya}, 
the manifold $M_{3,q}$ is the Seifert fibered space $(D,(3,1),(l-1,l))$. 
All the other $M_{p,q}$ (i.e., $p \ge 5$) are hyperbolic. 
In particular, $M_{5,5}$ is homeomorphic to 
the ``figure-8 knot sister manifold''. 
\end{corollary}

Theorem~\ref{thm:toroidal_kabaya} and the Corollary~\ref{cor:(-2,p,p)Tor} will be proved in Section~\ref{sec:toroidal}. 


\medskip

Our second result concerns Seifert fibered surgeries on such pretzel knots. 
For instance, $P(-2,3,7)$ is well-known 
for it is the first hyperbolic example, 
which admits non-trivial Seifert fibered surgeries~\cite{FintushelStern1980}. 
On the other hand, in the case where $p=q$, we obtain the following: 

\begin{theorem}
\label{thm:SF_IJ}
A pretzel knot $P(-2,p,p)$ with positive integers $p \ge5$ 
admits no Seifert fibered surgeries. 
\end{theorem}

This will be proved in Section 3 by applying 
a method developed in \cite{IchiharaJong2011} by the first two authors. 

\medskip

We note that our theorems together with known facts 
complete the classification 
of the exceptional surgeries on $P(-2,p,p)$ with $p \ge 5$. 
To see this, and also as a background, 
we recall some of known facts 
on exceptional surgeries on hyperbolic pretzel knots. 
Actually most of the following results concern Montesinos knots. 
However, for simplicity, we only deal with pretzel knots. 
See the original references for precise statements.

Wu showed that 
there are no reducible surgery on hyperbolic pretzel knots~\cite{Wu1996}, 
and also obtained a complete classification of 
toroidal surgeries on pretzel knots~\cite{Wu2006}. 
If a pretzel knot contains at most two non-integer tangles, 
then it is equivalent to a two-bridge knot, 
and then exceptional surgeries on such knots 
were completely determined in \cite{BrittenhamWu2001}. 
On the other hand, 
if a pretzel knot contains at least four non-integer tangles, 
then it was also shown by Wu~\cite{Wu1996} that 
such a pretzel knot admits no exceptional surgery. 
Furthermore, on pretzel knots, 
the first two authors gave a complete classification of surgeries 
yielding $3$-manifolds 
with cyclic or finite fundamental groups \cite{IchiharaJong2009}, and 
showed that there are no toroidal Seifert surgeries on pretzel knots 
other than the trefoil~\cite{IchiharaJong2010a}. 
Very recently, in \cite{Wu2009,Wu2010}, 
Wu gave several restrictions, in particular, he showed that if a hyperbolic pretzel knot of length three admits an atoroidal Seifert fibered surgery, then it is equivalent to 
$P(q_1,q_2,q_3,n)$ with $n=0,-1$ and, up to relabeling, 
$(|q_1|,|q_2|,|q_3|) = (2,|q_2|,|q_3|), (3,3,|q_3|), \mbox{ or } (3,4,5)$~\cite[Theorem 7.2]{Wu2010}.


\section{Toroidal surgeries}\label{sec:toroidal}

In this section, 
we give a proof of Theorem~\ref{thm:toroidal_kabaya} 
and Corollary~\ref{cor:(-2,p,p)Tor}. 

First of all we set up our definitions and notations. 

A \textit{pretzel knot} of type $(a_1, a_2, a_3)$ 
with integers $a_1, a_2, a_3$, 
denoted by $P(a_1, a_2, a_3)$, is defined as 
a knot admitting a diagram obtained by 
putting rational tangles of the forms $1/a_1, 1/a_2, 1/a_3$ together in a circle. 

From a given knot $K$ in $S^3$, 
we obtain a closed orientable 3-manifold 
by a \textit{Dehn surgery} on $K$ as follows: 
Remove the interior of a tubular neighborhood $N(K)$ of $K$, and 
glue solid torus back. 
The slope 
(i.e., the isotopy class of an unoriented non-trivial simple closed curve) 
on the peripheral torus $\partial N(K)$, 
which is identified with the meridian of the attached solid torus 
is called the \textit{surgery slope}. 
It is well-known that slopes on the torus $\partial N(K)$ are parameterized by $\mathbb{Q} \cup \{ 1/0 \}$ by using the standard meridian-longitude system for $K$.
Thus, when the surgery slope corresponds to $r \in \mathbb{Q} \cup \{1/0\}$, 
we call the Dehn surgery on $K$ along the surgery slope 
the \textit{$r$-Dehn surgery} or \textit{$r$-surgery} for brevity, 
and denote the obtained manifold by $K(r)$.

\subsection{Proof of Theorem~\ref{thm:toroidal_kabaya}}
Let $K$ be the $(-2,p,q)$-pretzel knot.
Put $K$ on a genus two surface $F$ which bounds two handlebodies in $S^3$ as shown in Figure \ref{fig:pretzel_on_surface}, 
denote the `outside' of $F$ by $V$ and the `inside' by $V'$.
The isotopy class on $\partial N(K)$ determined by the intersection $F \cap \partial N(K)$ is called the \textit{surface slope of $K$ with respect to $F$}.
Now the surface slope is $2(p+q)$.

\begin{figure}[htb]
\centering
\input{pretzel_on_surface.pstex_t}
\caption{}
\label{fig:pretzel_on_surface}
\end{figure}

The manifold obtained from a Dehn surgery on $K$ along the surface slope is described as follows.

\begin{lemma}[{\cite[Lemma 2.1]{dean}}]
Let $W$ (resp. $W'$) be the manifold obtained from $V$ (resp. $V'$) by attaching a $2$-handle along $K$
and $\overline{F} = (F - N(K)) \cup (D^2 \times \{0,1\})$.
Then the manifold obtained from Dehn surgery on $K$ with surface slope is homeomorphic to $W \cup_{\overline{F}} W'$.
\end{lemma}

The inside $V'$ contains a properly embedded one-holed Klein bottle whose boundary coincides with $K$, 
and $V'$ is homeomorphic to the regular neighborhood of the one-holed Klein bottle.
Therefore $W' = V' \cup_K (\textrm{2-handle})$ is a twisted $I$-bundle over the Klein bottle and $\overline{F}$ is a torus.
We will show that $W = V \cup_K (\textrm{2-handle})$ is obtained by a Dehn surgery on the chain-link with three components.

Instead of considering the $(-2,p,q)$-pretzel knots, we study the $3$-component link depicted in Figure \ref{fig:associated_link}, 
which is obtained from the $(-2,1,1)$-pretzel knot $K_0$ by adding two trivial components $T_1$ and $T_2$ encircling the two half-twisted strands respectively.
The $(-2,p,q)$-pretzel knot is obtained from the link by the $-1/k$-surgery along $T_1$ and the $-1/l$-surgery along $T_2$, where $p=2k+1$ and $q = 2l+1$ respectively.
Therefore the manifold $W$ is obtained from $V$ by attaching a $2$-handle along $K_0$ then doing the $-1/k$-surgery along $T_1$ and the $-1/l$-surgery along $T_2$. 
Let $W_0$ be the manifold obtained from $V - (N(T_1) \cup N(T_2) )$ by attaching a $2$-handle along $K_0$.
We will show that $W_0$ is homeomorphic to the exterior $N$ of the chain-link with three components, 
and describe the relation between peripheral curves of $W_0$ and $N$.
Here we fix the standard meridian and longitude on each component of $\partial N$ and we identify slopes with $\mathbb{Q} \cup \{\infty\}$.

\begin{figure}[htb]
\begin{center}
\input{associated_link.pstex_t}
\caption{}
\label{fig:associated_link}
\end{center}
\end{figure}

Take meridian disks $A$ and $B$ of the outside handlebody $V$ as shown in Figure \ref{fig:associated_link} and cut $V$ along $A$ and $B$.
Then remove the regular neighborhood of the two arcs corresponding to $T_1$ and $T_2$.
Cut the resulting manifold along the disks $C$ and $D$ as indicated in Figure \ref{fig:handle_decomp_1}.
This gives a handle decomposition of $W_0$ (Figure \ref{fig:handle_decomp_2}).
In Figure \ref{fig:handle_decomp_2}, we modify the regions encircled by dotted curves by the operations described in Figure \ref{fig:simplification_1}.  
Then we obtain a simplified handle decomposition of $W_0$ (Figures \ref{fig:handle_decomp_4} and \ref{fig:handle_decomp_5}).
We regard the diagram of the handle decomposition given in Figure \ref{fig:handle_decomp_5} as a trivalent graph, 
then taking the dual of this trivalent graph, we obtain a triangulation of the boundary of a $3$-ball (Figure \ref{fig:polyhedron}).
In this way we regard the handle decomposition of $W_0$ as a topological ideal polyhedral decomposition of $W_0$.
The ideal polyhedron further decomposed into 6 ideal tetrahedra.
By using SnapPea \cite{SnapPea}, we can check that $W_0$ is obtained from gluing 6 positively oriented ideal tetrahedra, therefore has a hyperbolic structure.
We can also check that there exists an isometry from $W_0$ to the exterior $N$ of the chain-link 
and we confirm that the isometry maps the slope $-1/k$ on $\partial W_0$ to the slope $-\frac{k+1}{k}$ on $\partial N$, and similarly for the slope $-1/l$.
This completes the (computer-aided) proof, but we also give an explicit homeomorphism between $W_0$ and $N$.

\begin{figure}[htb]
\centering
\input{handle_decomp_1.pstex_t}
\caption{}
\label{fig:handle_decomp_1}
\end{figure}

\begin{figure}[htb]
\centering
\input{handle_decomp_2.pstex_t}
\caption{}
\label{fig:handle_decomp_2}
\end{figure}

\begin{figure}[htb]
\centering
\input{simplification_1.pstex_t}

\vspace{30pt}
\input{simplification_2.pstex_t}
\caption{}
\label{fig:simplification_1}
\end{figure}


\begin{figure}[htb]
\centering
\input{handle_decomp_4.pstex_t}
\caption{}
\label{fig:handle_decomp_4}
\end{figure}

\begin{figure}[htb]
\centering
\input{handle_decomp_5.pstex_t}
\caption{}
\label{fig:handle_decomp_5}
\end{figure}

\begin{figure}[htb]
\centering
\input{polyhedron.pstex_t}
\caption{}
\label{fig:polyhedron}
\end{figure}

We decompose the chain-link exterior into two ``drums'' according to Section 6 of \cite{Thurston1978}.
Let $X$, $Y$ and $Z$ be the disks bounded by the components of the chain-link in the simplest way (Figure \ref{fig:chain_link}). 
Slicing the exterior $N$ along the disks, then we obtain a solid torus whose boundary is tiled by quadrilaterals (Figure \ref{fig:sliced_link_complement}).
This solid torus is decomposed into two drums and further into 6 tetrahedra as shown in Figure \ref{fig:splited_drums}. 
Glue together these 6 ideal tetrahedra along the faces which contain the double arrowed edges of Figure \ref{fig:splited_drums}, 
we obtain an ideal polyhedron with 12 faces (Figure \ref{fig:chain_polyhedron}). 
Since the gluing pattern of the ideal polyhedron is equivalent to the one given in Figure \ref{fig:polyhedron}, $W_0$ and $N$ are homeomorphic.

\begin{figure}[htb]
\centering
\input{chain_link_with_gen.pstex_t}
\caption{}
\label{fig:chain_link}
\end{figure}

\begin{figure}[htb]
\centering
\input{sliced_link_complement.pstex_t}
\caption{}
\label{fig:sliced_link_complement}
\end{figure}

\begin{figure}[htb]
\centering
\input{two_drums.pstex_t}
\hspace{30pt}
\input{splitted_drums.pstex_t}
\caption{}
\label{fig:splited_drums}
\end{figure}

\begin{figure}[htb]
\centering
\input{chain_polyhedron.pstex_t}
\caption{}
\label{fig:chain_polyhedron}
\end{figure}

Finally we observe the correspondence between peripheral curves of $W_0$ and $N$.
In the polyhedral decomposition of $W_0$ given in Figure \ref{fig:polyhedron}, 
a path $m_1$ from the face $C''$ to the face $c''$ and a path $l_1$ from the face $A$ to the face $a$ form a meridian and longitude pair of $T_1$.
Here the path corresponding to $m_1$ in $N$ is represented by a word of the form $p_y$ and the path corresponding to $l_1$ by $p_x p_z (p_y)^{-1}$, 
where $p_x$ (resp. $p_y$, $p_z$) is an element of the fundamental group of $N$ intersecting the disk $X$  (resp. $Y$, $Z$) at once (Figure \ref{fig:chain_link}).
On the other hand a meridian and longitude pair of the component corresponding to the disk $Y$ is represented by the words $p_y$ and $p_z p_x$.
Therefore the slope $p/q$ on $T_1$ is mapped to $\frac{p-q}{q}$, in particular $-1/k$ to $-\frac{k+1}{k}$.
By symmetry, the slope $-1/l$ on $T_2$ is mapped to $-\frac{l+1}{l}$. \qed

\subsection{Proof of Corollary~\ref{cor:(-2,p,p)Tor}}
Let $N$ be the complement of the chain-link with three components, also known as the \textit{magic manifold}. 
We denote the $p/q$- and $r/s$-Dehn filling of $N$ by $N(p/q, r/s)$.
Since any two components of $\partial N$ can be interchanged by an automorphism which preserves the peripheral structure, this notation makes sense. 

Then, by Theorem~\ref{thm:toroidal_kabaya}, 
the manifold $M_{p,q}$ is homeomorphic to 
$N(-\frac{k+1}{k}, -\frac{l+1}{l})$ where $p=2k+1$ and $q=2l+1$ with $k \leq l$. 

On the other hand, by the result of Martelli and Petronio \cite{martelli-petronio}, we know that 
$N(a/b,c/d)$ is hyperbolic except if one of the following occurs up to permutation:
\begin{itemize}
\item $a/b \in \{ \infty, -3,-2,-1,0\}$, 
\item $(a/b,c/d) \in \{ (1,1), (-4,-1/2), (-3/2,-5/2)\}$.
\end{itemize}
Thus we see that all the manifolds $M_{p,q}$ with $p \ge 5$ (i.e., $k \ge2$) are hyperbolic. 

Furthermore, when $p=3$, equivalently $k=1$, it is shown in \cite{martelli-petronio} that 
the manifold $N(-2, -\frac{l+1}{l})$ is homeomorphic to $(D,(3,1),(l-1,l))$. 

In particular, 
the manifold $M_{5,5}$ is homeomorphic to $N(-\frac{3}{2}, -\frac{3}{2})$, 
which is the ``figure-8 knot sister manifold''. 
Also note that, since $N(-3/2)$ is the Whitehead sister link ($(-2,3,8)$-pretzel link), 
the $M_{5,q}$ is obtained from the Whitehead sister link. \qed


\section{Seifert fibered surgeries}\label{sec:Seifert}

In this section we give a proof of Theorem~\ref{thm:SF_IJ}. 
Essentially the proof is on the same line as 
that for \cite[Proposition 3.7]{IchiharaJong2011}. 

\bigskip

\noindent
\textbf{Proof of Theorem~\ref{thm:SF_IJ}.} 
Let $K$ be a pretzel knot $P(-2,p,p)$ with a positive integer $p \ge5$. 
Assume for the contrary that 
$K$ admits a Seifert fibered surgery, i.e., 
$K(r)$ is Seifert fibered for some $r \in \mathbb{Q}$. 
Then by the results in \cite{FuterIshikawaKabayaMattmanShimokawa2009}, \cite{IchiharaJong2009}, and \cite{IchiharaJong2010a}, 
$K(r)$ must be 
a Seifert fibered manifold with a base orbifold $S^2$ 
having three singular fibers. 
In particular, $K(r)$ is atoroidal. 

First we give a restriction of the slope $r$ as follows. 

\begin{claim}
We have $r = 4 p \pm 1$. 
\end{claim}

\begin{proof}
First we see that $r$ must be an integer. 
Actually, if $K(r)$ is atoroidal Seifert fibered, 
then $r \in \mathbb{Z}$ unless $K$ is equivalent to 
one of the Montesinos knots of type 
$(1/3, \pm 1/3, \ast)$ or $(1/2, 1/3, \ast)$~\cite[Theorem 8.3]{Wu2009}. 
See \cite{Wu2009} for the precise statement. 

Next, we note that $K$ is a periodic knot 
with period two as shown in Figure~\ref{period2}. 
The factor knot $K'$ with respect to this cyclic period is 
equivalent to a torus knot $T(2,p)$. 
Then, since the following diagram commutes, 
by \cite[Lemma 3.8]{IchiharaJong2011}, 
originally observed in \cite{MiyazakiMotegi2002a}, 
$T(2,p)(r/2)$ must be homeomorphic to a lens space. 
$$\xymatrix{
S^3 \ar[d]_{r\text{-surgery}} \ar[r]^{/f} \ar@{}[dr]|\circlearrowleft & S^3/f=S^3 \ar@{>}[dr]^{r/2\text{-surgery}}&~ \\
K(r) \ar[r]^{/\bar{f}} & K(r)/\bar{f} & \hspace{-35pt} \cong K'(r/2)
}$$
Then we have $r/2 = 2 p \pm 1/2$ by the classification of Dehn surgeries on torus knots due to Moser~\cite{Moser1971}. 
Since $r \in \mathbb{Z}$, we have $r=4 p \pm 1$. 
\begin{figure}[htb]
\centering
\includegraphics[width=0.4\textwidth]{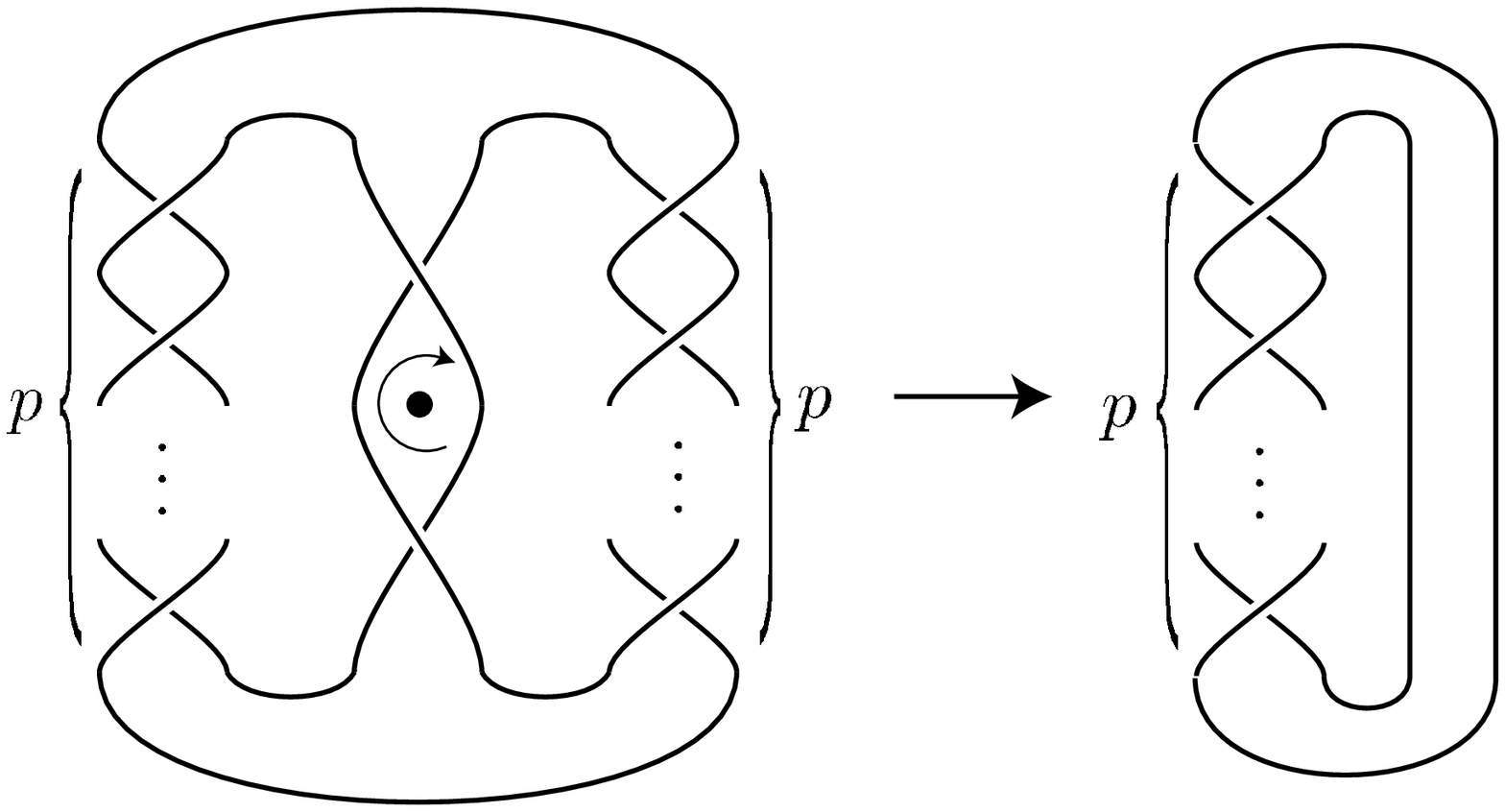}
\caption{}\label{period2}
\end{figure}
\end{proof}

Next we apply the \textit{Montesinos trick}, originally introduced in \cite{Montesinos1975}. 
Set an axis 
which induces a strong inversion of $K$ as shown in Figure~\ref{Kp}. 
Then, applying the Montesinos trick, 
we see that 
the surgered manifold $K(4p \pm1)$ is homeomorphic to 
the double branched cover of $S^3$ branched along 
the knot $K_{p \pm}$ depicted in Figure~\ref{Kp}. 

\begin{figure}[htb]
\centering
\includegraphics[width=0.5\textwidth]{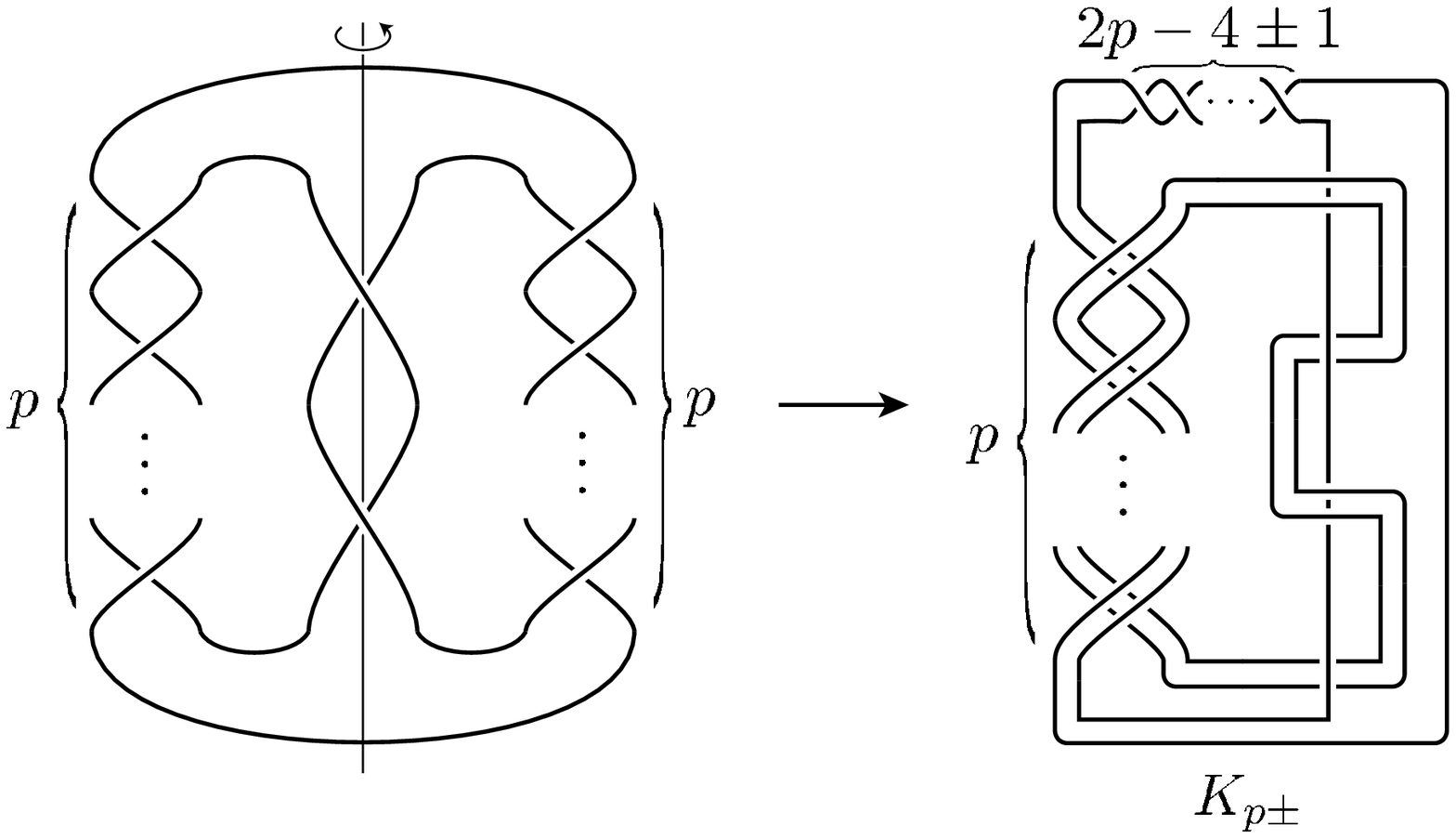}
\caption{}\label{Kp}
\end{figure}

By virtue of the next claim, to complete the proof of the theorem, 
it suffices to show that $K_{p \pm}$ is neither a Montesinos knot nor a torus knot.

\begin{claim}\label{clm:criterion}
The knot $K_{p \pm}$ must be a Montesinos knot or a torus knot
\end{claim}
\begin{proof}
For a strongly invertible hyperbolic knot $K$ and a rational number $r$, 
if $K(r)$ is a Seifert fibered manifold with the base orbifold $S^2$, 
then the link $L_r$ obtained by applying the Montesinos trick to $K(r)$, 
i.e., the link $L_r$ satisfying that the double branched cover of $S^3$ branched along $L_r$ is homeomorphic to $K(r)$, 
is equivalent to a Montesinos link or a Seifert link~\cite[Proposition 2.1]{IchiharaJong2011}. 
Here a link is said to be {\it Seifert} if its exterior is Seifert fibered. 
Also see \cite{MiyazakiMotegi2002a} and \cite{Motegi2003}. 

Note that $K_{p \pm}$ is a knot since $r=4p \pm1$ is odd. 
Since Seifert links are completely classified in \cite{BurdeMurasugi1970} (see also~\cite[Proposition 7.3]{EisenbudNeumann1985}), 
by this classification, we see that $K_{p \pm}$ is Seifert if and only if $K_{p \pm}$ is a torus knot. 
\end{proof}

Applying the criterion due to the first two authors~\cite{IchiharaJong2011}, 
which uses the Rasmussen invariant, 
we show that $K_{p \pm}$ is not a Montesinos knot as follows. 

\begin{claim}\label{clm:Montesinos}
The knot $K_{p \pm}$ is not a Montesinos knot. 
\end{claim}
\begin{proof}
We here apply the following fact: 
If $| s(K_{p\pm}) + \sigma(K_{p\pm})| \ge 4$, then $K_{p\pm}$ is not a Montesinos knot~\cite[Criterion 2.5]{IchiharaJong2011}. 
Here $s(K)$ denotes the \textit{Rasmussen invariant} for a knot $K$ and 
$\sigma(K)$ the \textit{signature} of a knot $K$. 


Now we need to calculate or estimate $s(K_{p\pm})$ and $\sigma(K_{p\pm})$. 

First we estimate the Rasmussen invariant $s(K_{p\pm})$ by using the following inequality obtained in \cite{Plamenevskaya2006} and \cite{Shumakovitch2007}. 
For a knot $K$ and a diagram $D$ of $K$, 
we have 
\begin{align}\label{eq:Bennequin}
s(K) \ge w(D) - O(D) +1 ,
\end{align}
where $w(D)$ denotes the writhe of $D$ and $O(D)$ denotes the number of Seifert circles of $D$. 
Applying this inequality to the diagram shown in Figure~\ref{Kp}, we have
\begin{align*}
s(K_{p\pm}) &\ge (4p - 8 + 2p -4 \pm 1) - 4 +1\\ 
&= 6p -15 \pm1.
\end{align*}

Next we calculate the signature $\sigma(K_{p\pm})$ by using the method due to Gordon and Litherland~\cite{GordonLitherland1978}. 
As shown in Figure~\ref{fig:surface}, $K_{p\pm}$ bounds a non-orientable surface $V_{p \pm}$ such that the first Betti number of $V_{p \pm}$ is equal to three. 
Take the loops $l_1$, $l_2$, and $l_3$ on $V_{p \pm}$, which form a basis of $H_1(V_{p \pm})$. 
Then a bilinear form $\mathcal{G}_{V_{p \pm}} : H_1(V_{p \pm}) \times H_1(V_{p \pm}) \rightarrow \mathbb{Z}$ introduced in \cite[Section 2]{GordonLitherland1978} is presented by the following matrix: 
\begin{align*}
\left(\begin{matrix}
4p -4 \pm1 & 0 & 2 \\ 
0 & 1 & 1 \\ 
2 & 1 & 0
\end{matrix}
\right).
\end{align*}

Since $p \ge 5$, we see that $\sign \mathcal{G}_{V_{p \pm}} = 1$, where $\sign \mathcal{G}_{V_{p \pm}}$ denotes the signature of $\mathcal{G}_{V_{p \pm}}$. 
Furthermore, by considering the boundary slope of $V_{p \pm}$, 
the normal Euler number of $V_{p}$ (see \cite[Section 3]{GordonLitherland1978}), denoted by $e(V_{p \pm})$, is 
shown to be $-8p +16 \mp 2$. 
Then by \cite[Corollary 5]{GordonLitherland1978}, we have 
\begin{align*}
\sigma(K_{p\pm}) &= \sign \mathcal{G}_{V_{p \pm}} + e(V_{p \pm})/2\\
&= 1 + (-8p +16 \mp 2)/2 \\
&= -4p + 9 \mp 1.
\end{align*}

\begin{figure}[htb]
\centering
\includegraphics[width=.23\textwidth]{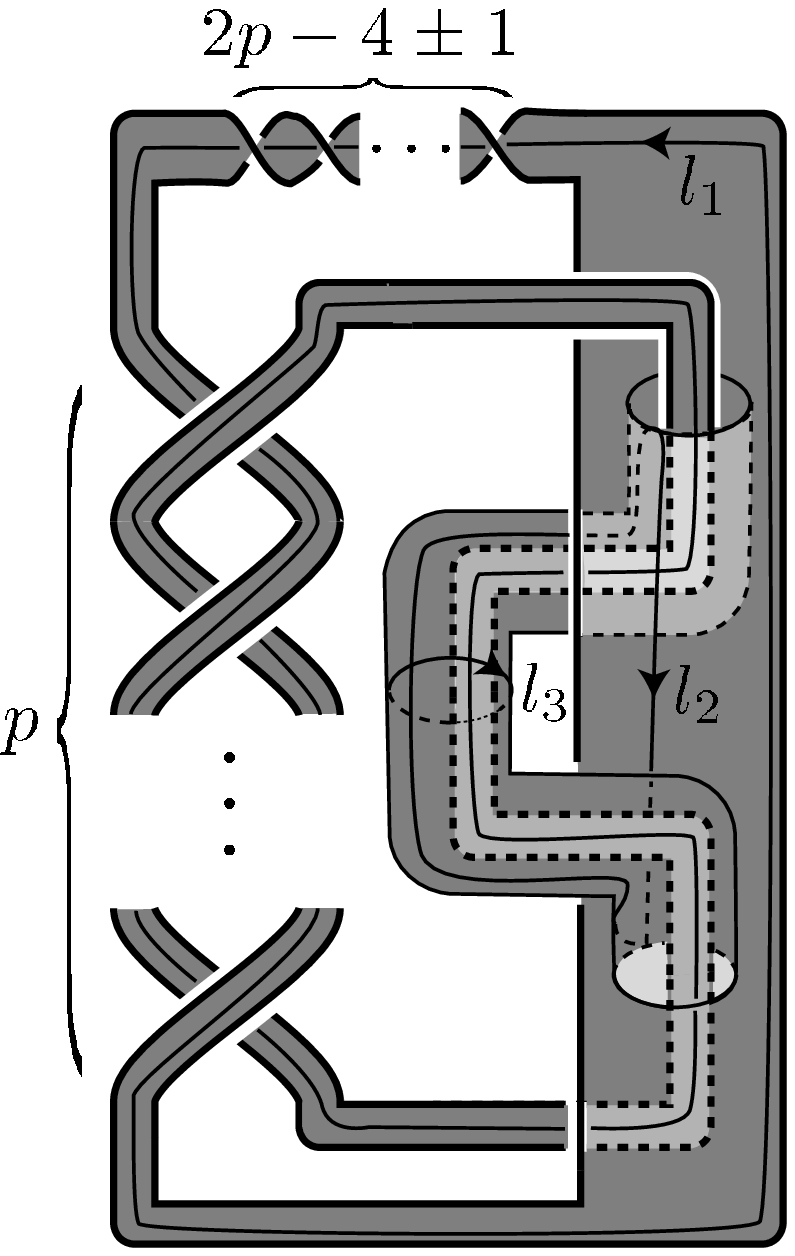}
\caption{}\label{fig:surface}
\end{figure}

Consequently, together with the assumption that $p \ge 5$, we obtain the following: 
\begin{align*}
|s(K_{p\pm}) + \sigma(K_{p\pm})| &\ge s(K_{p\pm}) + \sigma(K_{p\pm})\\
& \ge (6p -15 \pm1) + (-4p+ 9 \mp 1)\\
&= 2p -6 \\
&\ge 4. 
\end{align*}
Thus $K_{p \pm}$ is not a Montesinos knot. 
\end{proof}

Finally we show that $K_{p \pm}$ is not a torus knot as follows. 

\begin{claim}\label{clm:torus}
The knot $K_{p \pm}$ is not a torus knot. 
\end{claim}

\begin{proof}
Suppose that $K_{p\pm}$ is a torus knot. 
As shown in Figure~\ref{Kp}, the knot $K_{p\pm}$ is represented as a closure of a four-braid, 
the braid index of $K_{p\pm}$ is at most four. 
Since $s(K_{p\pm}) \ne - \sigma(K_{p\pm})$ as in the proof of Claim~\ref{clm:Montesinos}, $K_{p\pm}$ is non-alternating. 
Actually, if a knot $K$ is alternating, then we have $s(K) = -\sigma(K)$~\cite[Theorem 3]{Rasmussen2004}. 
Hence the braid index of $K_{p\pm}$ is three or four. 

Then we see that $K_{p \pm } = T(4, 4p \pm 1)$ as follows: 
For a knot $K$, let $\det(K)$ be the {\it determinant} of $K$ and $\Delta_K(t)$ the {\it Alexander polynomial} of $K$. 
Note that we have $\det(K) = \Delta_K(-1)$ and it also coincides with the order of the first homology group of the double branched covering space of $S^3$ branched along $K$ (see for example \cite{KawauchiBook} or \cite{RolfsenBook}). 

Since $K(r)$ is the double branched cover of $S^3$ branched along $K_{p\pm}$, 
we have $\det(K_{p\pm}) = | H_1 (K(r)) | = r = 4p \pm 1$. 
On the other hand, since $\Delta_{T(x,y)}(t) = (t^{xy}-1)(t-1)/(t^x -1)(t^y -1)$, we have $\det (T(3,x)) = 1$ or $3$, and $\det (T(4,x)) = x$. 
Since $4p \pm 1 \ge 19$, $K_{p\pm}$ is equivalent to $T(4,x)$ and we also have $x = 4p \pm 1$. 

Next we consider the Rasmussen invariant of $K_{p \pm }$ and $T(4,4p \pm 1)$. 
For a knot $K$, we denote by $K^*$ the mirror image of $K$. 
By the inequality~\eqref{eq:Bennequin}, we have
\begin{align*}
s(K_{p\pm}^*) &\ge -(4p -8 + 2p -4 \mp 1) -4 +1\\
&= -6p +9 \mp 1. 
\end{align*}
Since $s(K^*) = -s(K)$ holds for a knot $K$~\cite[Theorem 2]{Rasmussen2004}, we have 
$$s(K_{p\pm}) \le 6p -9 \pm 1. $$

On the other hand, since for a positive knot $K$, we have $s(K) = 2g(K)$~\cite[Theorem 4]{Rasmussen2004} and $2g(T(4,4q \pm1)) = 12p \pm 3 -3$, we have 
$$s(T(4,4p \pm1)) = 12p \pm 3 -3. $$
Therefore we have $ s(K_{p\pm}) < s(T(4,4p \pm1))$ and thus, $K_{p\pm} \ne T(4,4p \pm1)$. 
A contradiction occurs. 
\end{proof}

Now, by Claim~\ref{clm:criterion}, Claim~\ref{clm:Montesinos} and Claim~\ref{clm:torus}, 
we have a contradiction, and complete the proof of Theorem~\ref{thm:SF_IJ}.
\qed

\begin{ack}
The authors would like to thank Professor Akira Yasuhara for helpful comments on a calculation of the signature of a knot. 
The first author is partially supported by 
Grant-in-Aid for Young Scientists (B), No.~20740039,
Ministry of Education, Culture, Sports, Science and Technology, Japan.
The second author is partially supported by Grant-in-Aid for Research Activity Start-up, No.~22840037, Japan Society for the Promotion of Science.
\end{ack}


\end{document}

%% file: pretzel_on_surface.pstex_t
\begin{picture}(0,0)%
\includegraphics{pretzel_on_surface.pstex}%
\end{picture}%
\setlength{\unitlength}{987sp}%
\begingroup\makeatletter\ifx\SetFigFont\undefined%
\gdef\SetFigFont#1#2#3#4#5{%
  \reset@font\fontsize{#1}{#2pt}%
  \fontfamily{#3}\fontseries{#4}\fontshape{#5}%
  \selectfont}%
\fi\endgroup%
\begin{picture}(9219,7824)(1147,-8173)
\put(10351,-4186){\makebox(0,0)[lb]{\smash{{\SetFigFont{10}{12.0}{\rmdefault}{\mddefault}{\updefault}{\color[rgb]{0,0,0}$q$}%
}}}}
\put(6751,-4186){\makebox(0,0)[lb]{\smash{{\SetFigFont{10}{12.0}{\rmdefault}{\mddefault}{\updefault}{\color[rgb]{0,0,0}$p$}%
}}}}
\end{picture}%

%% file: associated_link.pstex_t
\begin{picture}(0,0)%
\includegraphics{associated_link.pstex}%
\end{picture}%
\setlength{\unitlength}{987sp}%
\begingroup\makeatletter\ifx\SetFigFont\undefined%
\gdef\SetFigFont#1#2#3#4#5{%
  \reset@font\fontsize{#1}{#2pt}%
  \fontfamily{#3}\fontseries{#4}\fontshape{#5}%
  \selectfont}%
\fi\endgroup%
\begin{picture}(10680,7824)(-14,-8173)
\put(7201,-4561){\makebox(0,0)[lb]{\smash{{\SetFigFont{10}{12.0}{\rmdefault}{\mddefault}{\updefault}{\color[rgb]{0,0,0}$B$}%
}}}}
\put(3601,-4561){\makebox(0,0)[lb]{\smash{{\SetFigFont{10}{12.0}{\rmdefault}{\mddefault}{\updefault}{\color[rgb]{0,0,0}$A$}%
}}}}
\put(10651,-5611){\makebox(0,0)[lb]{\smash{{\SetFigFont{10}{12.0}{\rmdefault}{\mddefault}{\updefault}{\color[rgb]{0,0,0}$T_2$}%
}}}}
\put(  1,-3211){\makebox(0,0)[lb]{\smash{{\SetFigFont{10}{12.0}{\rmdefault}{\mddefault}{\updefault}{\color[rgb]{0,0,0}$T_1$}%
}}}}
\end{picture}%

%% file: handle_decomp_1.pstex_t
\begin{picture}(0,0)%
\includegraphics{handle_decomp_1.pstex}%
\end{picture}%
\setlength{\unitlength}{987sp}%
\begingroup\makeatletter\ifx\SetFigFont\undefined%
\gdef\SetFigFont#1#2#3#4#5{%
  \reset@font\fontsize{#1}{#2pt}%
  \fontfamily{#3}\fontseries{#4}\fontshape{#5}%
  \selectfont}%
\fi\endgroup%
\begin{picture}(15108,7908)(-1853,-8815)
\put(-449,-5311){\makebox(0,0)[lb]{\smash{{\SetFigFont{8}{9.6}{\rmdefault}{\mddefault}{\updefault}{\color[rgb]{0,0,0}$2$}%
}}}}
\put(4276,-5311){\makebox(0,0)[lb]{\smash{{\SetFigFont{8}{9.6}{\rmdefault}{\mddefault}{\updefault}{\color[rgb]{0,0,0}$2$}%
}}}}
\put(751,-4711){\makebox(0,0)[lb]{\smash{{\SetFigFont{8}{9.6}{\rmdefault}{\mddefault}{\updefault}{\color[rgb]{0,0,0}$3$}%
}}}}
\put(3076,-4711){\makebox(0,0)[lb]{\smash{{\SetFigFont{8}{9.6}{\rmdefault}{\mddefault}{\updefault}{\color[rgb]{0,0,0}$3$}%
}}}}
\put(10351,-3586){\makebox(0,0)[lb]{\smash{{\SetFigFont{8}{9.6}{\rmdefault}{\mddefault}{\updefault}{\color[rgb]{0,0,0}$3$}%
}}}}
\put(4276,-4411){\makebox(0,0)[lb]{\smash{{\SetFigFont{8}{9.6}{\rmdefault}{\mddefault}{\updefault}{\color[rgb]{0,0,0}$1$}%
}}}}
\put(-449,-4411){\makebox(0,0)[lb]{\smash{{\SetFigFont{8}{9.6}{\rmdefault}{\mddefault}{\updefault}{\color[rgb]{0,0,0}$1$}%
}}}}
\put(6751,-2911){\makebox(0,0)[lb]{\smash{{\SetFigFont{8}{9.6}{\rmdefault}{\mddefault}{\updefault}{\color[rgb]{0,0,0}$2$}%
}}}}
\put(6751,-4411){\makebox(0,0)[lb]{\smash{{\SetFigFont{8}{9.6}{\rmdefault}{\mddefault}{\updefault}{\color[rgb]{0,0,0}$1$}%
}}}}
\put(7876,-3511){\makebox(0,0)[lb]{\smash{{\SetFigFont{8}{9.6}{\rmdefault}{\mddefault}{\updefault}{\color[rgb]{0,0,0}$3$}%
}}}}
\put(11551,-4111){\makebox(0,0)[lb]{\smash{{\SetFigFont{8}{9.6}{\rmdefault}{\mddefault}{\updefault}{\color[rgb]{0,0,0}$1$}%
}}}}
\put(11551,-2911){\makebox(0,0)[lb]{\smash{{\SetFigFont{8}{9.6}{\rmdefault}{\mddefault}{\updefault}{\color[rgb]{0,0,0}$2$}%
}}}}
\put(301,-3961){\makebox(0,0)[lb]{\smash{{\SetFigFont{10}{12.0}{\rmdefault}{\mddefault}{\updefault}{\color[rgb]{0,0,0}$a$}%
}}}}
\put(3451,-3961){\makebox(0,0)[lb]{\smash{{\SetFigFont{10}{12.0}{\rmdefault}{\mddefault}{\updefault}{\color[rgb]{0,0,0}$A$}%
}}}}
\put(7276,-4711){\makebox(0,0)[lb]{\smash{{\SetFigFont{10}{12.0}{\rmdefault}{\mddefault}{\updefault}{\color[rgb]{0,0,0}$B$}%
}}}}
\put(10801,-4711){\makebox(0,0)[lb]{\smash{{\SetFigFont{10}{12.0}{\rmdefault}{\mddefault}{\updefault}{\color[rgb]{0,0,0}$b$}%
}}}}
\put(1801,-3286){\makebox(0,0)[lb]{\smash{{\SetFigFont{10}{12.0}{\rmdefault}{\mddefault}{\updefault}{\color[rgb]{0,0,0}$C$}%
}}}}
\put(9001,-5986){\makebox(0,0)[lb]{\smash{{\SetFigFont{10}{12.0}{\rmdefault}{\mddefault}{\updefault}{\color[rgb]{0,0,0}$D$}%
}}}}
\end{picture}%

%% file: handle_decomp_2.pstex_t
\begin{picture}(0,0)%
\includegraphics{handle_decomp_2.pstex}%
\end{picture}%
\setlength{\unitlength}{1184sp}%
\begingroup\makeatletter\ifx\SetFigFont\undefined%
\gdef\SetFigFont#1#2#3#4#5{%
  \reset@font\fontsize{#1}{#2pt}%
  \fontfamily{#3}\fontseries{#4}\fontshape{#5}%
  \selectfont}%
\fi\endgroup%
\begin{picture}(13308,7908)(-53,-8515)
\put(4876,-5461){\makebox(0,0)[lb]{\smash{{\SetFigFont{6}{7.2}{\rmdefault}{\mddefault}{\updefault}{\color[rgb]{0,0,0}$3$}%
}}}}
\put(5701,-5461){\makebox(0,0)[lb]{\smash{{\SetFigFont{6}{7.2}{\rmdefault}{\mddefault}{\updefault}{\color[rgb]{0,0,0}$2$}%
}}}}
\put(5701,-4636){\makebox(0,0)[lb]{\smash{{\SetFigFont{6}{7.2}{\rmdefault}{\mddefault}{\updefault}{\color[rgb]{0,0,0}$1$}%
}}}}
\put(5026,-4261){\makebox(0,0)[lb]{\smash{{\SetFigFont{6}{7.2}{\rmdefault}{\mddefault}{\updefault}{\color[rgb]{0,0,0}$4$}%
}}}}
\put(5551,-4261){\makebox(0,0)[lb]{\smash{{\SetFigFont{6}{7.2}{\rmdefault}{\mddefault}{\updefault}{\color[rgb]{0,0,0}$5$}%
}}}}
\put(2026,-4261){\makebox(0,0)[lb]{\smash{{\SetFigFont{6}{7.2}{\rmdefault}{\mddefault}{\updefault}{\color[rgb]{0,0,0}$4$}%
}}}}
\put(1276,-5236){\makebox(0,0)[lb]{\smash{{\SetFigFont{6}{7.2}{\rmdefault}{\mddefault}{\updefault}{\color[rgb]{0,0,0}$2$}%
}}}}
\put(1276,-4636){\makebox(0,0)[lb]{\smash{{\SetFigFont{6}{7.2}{\rmdefault}{\mddefault}{\updefault}{\color[rgb]{0,0,0}$1$}%
}}}}
\put(2101,-5461){\makebox(0,0)[lb]{\smash{{\SetFigFont{6}{7.2}{\rmdefault}{\mddefault}{\updefault}{\color[rgb]{0,0,0}$3$}%
}}}}
\put(5026,-4786){\makebox(0,0)[lb]{\smash{{\SetFigFont{9}{10.8}{\rmdefault}{\mddefault}{\updefault}{\color[rgb]{0,0,0}$A$}%
}}}}
\put(1501,-4261){\makebox(0,0)[lb]{\smash{{\SetFigFont{6}{7.2}{\rmdefault}{\mddefault}{\updefault}{\color[rgb]{0,0,0}$5$}%
}}}}
\put(1801,-4861){\makebox(0,0)[lb]{\smash{{\SetFigFont{9}{10.8}{\rmdefault}{\mddefault}{\updefault}{\color[rgb]{0,0,0}$a$}%
}}}}
\put(7276,-3436){\makebox(0,0)[lb]{\smash{{\SetFigFont{6}{7.2}{\rmdefault}{\mddefault}{\updefault}{\color[rgb]{0,0,0}$1$}%
}}}}
\put(7351,-3886){\makebox(0,0)[lb]{\smash{{\SetFigFont{6}{7.2}{\rmdefault}{\mddefault}{\updefault}{\color[rgb]{0,0,0}$5$}%
}}}}
\put(7951,-3886){\makebox(0,0)[lb]{\smash{{\SetFigFont{6}{7.2}{\rmdefault}{\mddefault}{\updefault}{\color[rgb]{0,0,0}$4$}%
}}}}
\put(7276,-2611){\makebox(0,0)[lb]{\smash{{\SetFigFont{6}{7.2}{\rmdefault}{\mddefault}{\updefault}{\color[rgb]{0,0,0}$2$}%
}}}}
\put(8101,-2611){\makebox(0,0)[lb]{\smash{{\SetFigFont{6}{7.2}{\rmdefault}{\mddefault}{\updefault}{\color[rgb]{0,0,0}$3$}%
}}}}
\put(11776,-3436){\makebox(0,0)[lb]{\smash{{\SetFigFont{6}{7.2}{\rmdefault}{\mddefault}{\updefault}{\color[rgb]{0,0,0}$1$}%
}}}}
\put(10876,-2611){\makebox(0,0)[lb]{\smash{{\SetFigFont{6}{7.2}{\rmdefault}{\mddefault}{\updefault}{\color[rgb]{0,0,0}$3$}%
}}}}
\put(10951,-3886){\makebox(0,0)[lb]{\smash{{\SetFigFont{6}{7.2}{\rmdefault}{\mddefault}{\updefault}{\color[rgb]{0,0,0}$4$}%
}}}}
\put(11626,-3886){\makebox(0,0)[lb]{\smash{{\SetFigFont{6}{7.2}{\rmdefault}{\mddefault}{\updefault}{\color[rgb]{0,0,0}$5$}%
}}}}
\put(11776,-2836){\makebox(0,0)[lb]{\smash{{\SetFigFont{6}{7.2}{\rmdefault}{\mddefault}{\updefault}{\color[rgb]{0,0,0}$2$}%
}}}}
\put(3451,-3736){\makebox(0,0)[lb]{\smash{{\SetFigFont{9}{10.8}{\rmdefault}{\mddefault}{\updefault}{\color[rgb]{0,0,0}$C$}%
}}}}
\put(9376,-4636){\makebox(0,0)[lb]{\smash{{\SetFigFont{9}{10.8}{\rmdefault}{\mddefault}{\updefault}{\color[rgb]{0,0,0}$D$}%
}}}}
\put(9451,-5911){\makebox(0,0)[lb]{\smash{{\SetFigFont{9}{10.8}{\rmdefault}{\mddefault}{\updefault}{\color[rgb]{0,0,0}$d$}%
}}}}
\put(3526,-2311){\makebox(0,0)[lb]{\smash{{\SetFigFont{9}{10.8}{\rmdefault}{\mddefault}{\updefault}{\color[rgb]{0,0,0}$c$}%
}}}}
\put(11176,-3286){\makebox(0,0)[lb]{\smash{{\SetFigFont{9}{10.8}{\rmdefault}{\mddefault}{\updefault}{\color[rgb]{0,0,0}$b$}%
}}}}
\put(7576,-3286){\makebox(0,0)[lb]{\smash{{\SetFigFont{9}{10.8}{\rmdefault}{\mddefault}{\updefault}{\color[rgb]{0,0,0}$B$}%
}}}}
\end{picture}%

%% file: simplification_1.pstex_t
\begin{picture}(0,0)%
\includegraphics{simplification_1.pstex}%
\end{picture}%
\setlength{\unitlength}{789sp}%
\begingroup\makeatletter\ifx\SetFigFont\undefined%
\gdef\SetFigFont#1#2#3#4#5{%
  \reset@font\fontsize{#1}{#2pt}%
  \fontfamily{#3}\fontseries{#4}\fontshape{#5}%
  \selectfont}%
\fi\endgroup%
\begin{picture}(13974,4830)(1789,-6376)
\put(12301,-4636){\makebox(0,0)[lb]{\smash{{\SetFigFont{8}{9.6}{\rmdefault}{\mddefault}{\updefault}{\color[rgb]{0,0,0}$A$}%
}}}}
\put(13651,-2086){\makebox(0,0)[lb]{\smash{{\SetFigFont{6}{7.2}{\rmdefault}{\mddefault}{\updefault}{\color[rgb]{0,0,0}$1$}%
}}}}
\put(13876,-4786){\makebox(0,0)[lb]{\smash{{\SetFigFont{6}{7.2}{\rmdefault}{\mddefault}{\updefault}{\color[rgb]{0,0,0}$4$}%
}}}}
\put(13726,-5686){\makebox(0,0)[lb]{\smash{{\SetFigFont{6}{7.2}{\rmdefault}{\mddefault}{\updefault}{\color[rgb]{0,0,0}$5$}%
}}}}
\put(4351,-4636){\makebox(0,0)[lb]{\smash{{\SetFigFont{8}{9.6}{\rmdefault}{\mddefault}{\updefault}{\color[rgb]{0,0,0}$A$}%
}}}}
\put(5701,-2086){\makebox(0,0)[lb]{\smash{{\SetFigFont{6}{7.2}{\rmdefault}{\mddefault}{\updefault}{\color[rgb]{0,0,0}$1$}%
}}}}
\put(5926,-3886){\makebox(0,0)[lb]{\smash{{\SetFigFont{6}{7.2}{\rmdefault}{\mddefault}{\updefault}{\color[rgb]{0,0,0}$3$}%
}}}}
\put(5926,-4786){\makebox(0,0)[lb]{\smash{{\SetFigFont{6}{7.2}{\rmdefault}{\mddefault}{\updefault}{\color[rgb]{0,0,0}$4$}%
}}}}
\put(5776,-5686){\makebox(0,0)[lb]{\smash{{\SetFigFont{6}{7.2}{\rmdefault}{\mddefault}{\updefault}{\color[rgb]{0,0,0}$5$}%
}}}}
\put(5851,-2986){\makebox(0,0)[lb]{\smash{{\SetFigFont{6}{7.2}{\rmdefault}{\mddefault}{\updefault}{\color[rgb]{0,0,0}$2$}%
}}}}
\end{picture}%

%% file: simplification_2.pstex_t
\begin{picture}(0,0)%
\includegraphics{simplification_2.pstex}%
\end{picture}%
\setlength{\unitlength}{789sp}%
\begingroup\makeatletter\ifx\SetFigFont\undefined%
\gdef\SetFigFont#1#2#3#4#5{%
  \reset@font\fontsize{#1}{#2pt}%
  \fontfamily{#3}\fontseries{#4}\fontshape{#5}%
  \selectfont}%
\fi\endgroup%
\begin{picture}(13974,4830)(1789,-6376)
\put(12301,-2836){\makebox(0,0)[lb]{\smash{{\SetFigFont{8}{9.6}{\rmdefault}{\mddefault}{\updefault}{\color[rgb]{0,0,0}$C'$}%
}}}}
\put(4351,-4261){\makebox(0,0)[lb]{\smash{{\SetFigFont{8}{9.6}{\rmdefault}{\mddefault}{\updefault}{\color[rgb]{0,0,0}$C$}%
}}}}
\put(12301,-5536){\makebox(0,0)[lb]{\smash{{\SetFigFont{8}{9.6}{\rmdefault}{\mddefault}{\updefault}{\color[rgb]{0,0,0}$C''$}%
}}}}
\end{picture}%

%% file: handle_decomp_4.pstex_t
\begin{picture}(0,0)%
\includegraphics{handle_decomp_4.pstex}%
\end{picture}%
\setlength{\unitlength}{1184sp}%
\begingroup\makeatletter\ifx\SetFigFont\undefined%
\gdef\SetFigFont#1#2#3#4#5{%
  \reset@font\fontsize{#1}{#2pt}%
  \fontfamily{#3}\fontseries{#4}\fontshape{#5}%
  \selectfont}%
\fi\endgroup%
\begin{picture}(13308,7908)(-53,-8515)
\put(5701,-4636){\makebox(0,0)[lb]{\smash{{\SetFigFont{6}{7.2}{\rmdefault}{\mddefault}{\updefault}{\color[rgb]{0,0,0}$1$}%
}}}}
\put(5026,-4261){\makebox(0,0)[lb]{\smash{{\SetFigFont{6}{7.2}{\rmdefault}{\mddefault}{\updefault}{\color[rgb]{0,0,0}$4$}%
}}}}
\put(5551,-4261){\makebox(0,0)[lb]{\smash{{\SetFigFont{6}{7.2}{\rmdefault}{\mddefault}{\updefault}{\color[rgb]{0,0,0}$5$}%
}}}}
\put(2026,-4261){\makebox(0,0)[lb]{\smash{{\SetFigFont{6}{7.2}{\rmdefault}{\mddefault}{\updefault}{\color[rgb]{0,0,0}$4$}%
}}}}
\put(1276,-4636){\makebox(0,0)[lb]{\smash{{\SetFigFont{6}{7.2}{\rmdefault}{\mddefault}{\updefault}{\color[rgb]{0,0,0}$1$}%
}}}}
\put(5026,-4786){\makebox(0,0)[lb]{\smash{{\SetFigFont{9}{10.8}{\rmdefault}{\mddefault}{\updefault}{\color[rgb]{0,0,0}$A$}%
}}}}
\put(1501,-4261){\makebox(0,0)[lb]{\smash{{\SetFigFont{6}{7.2}{\rmdefault}{\mddefault}{\updefault}{\color[rgb]{0,0,0}$5$}%
}}}}
\put(1801,-4861){\makebox(0,0)[lb]{\smash{{\SetFigFont{9}{10.8}{\rmdefault}{\mddefault}{\updefault}{\color[rgb]{0,0,0}$a$}%
}}}}
\put(7276,-3436){\makebox(0,0)[lb]{\smash{{\SetFigFont{6}{7.2}{\rmdefault}{\mddefault}{\updefault}{\color[rgb]{0,0,0}$1$}%
}}}}
\put(7351,-3886){\makebox(0,0)[lb]{\smash{{\SetFigFont{6}{7.2}{\rmdefault}{\mddefault}{\updefault}{\color[rgb]{0,0,0}$5$}%
}}}}
\put(7951,-3886){\makebox(0,0)[lb]{\smash{{\SetFigFont{6}{7.2}{\rmdefault}{\mddefault}{\updefault}{\color[rgb]{0,0,0}$4$}%
}}}}
\put(11776,-3436){\makebox(0,0)[lb]{\smash{{\SetFigFont{6}{7.2}{\rmdefault}{\mddefault}{\updefault}{\color[rgb]{0,0,0}$1$}%
}}}}
\put(10951,-3886){\makebox(0,0)[lb]{\smash{{\SetFigFont{6}{7.2}{\rmdefault}{\mddefault}{\updefault}{\color[rgb]{0,0,0}$4$}%
}}}}
\put(11626,-3886){\makebox(0,0)[lb]{\smash{{\SetFigFont{6}{7.2}{\rmdefault}{\mddefault}{\updefault}{\color[rgb]{0,0,0}$5$}%
}}}}
\put(11176,-3286){\makebox(0,0)[lb]{\smash{{\SetFigFont{9}{10.8}{\rmdefault}{\mddefault}{\updefault}{\color[rgb]{0,0,0}$b$}%
}}}}
\put(7651,-3286){\makebox(0,0)[lb]{\smash{{\SetFigFont{9}{10.8}{\rmdefault}{\mddefault}{\updefault}{\color[rgb]{0,0,0}$B$}%
}}}}
\put(3976,-3661){\makebox(0,0)[lb]{\smash{{\SetFigFont{8}{9.6}{\rmdefault}{\mddefault}{\updefault}{\color[rgb]{0,0,0}$C''$}%
}}}}
\put(4051,-2311){\makebox(0,0)[lb]{\smash{{\SetFigFont{8}{9.6}{\rmdefault}{\mddefault}{\updefault}{\color[rgb]{0,0,0}$c''$}%
}}}}
\put(2701,-3661){\makebox(0,0)[lb]{\smash{{\SetFigFont{8}{9.6}{\rmdefault}{\mddefault}{\updefault}{\color[rgb]{0,0,0}$C'$}%
}}}}
\put(2776,-2311){\makebox(0,0)[lb]{\smash{{\SetFigFont{8}{9.6}{\rmdefault}{\mddefault}{\updefault}{\color[rgb]{0,0,0}$c'$}%
}}}}
\put(8626,-4561){\makebox(0,0)[lb]{\smash{{\SetFigFont{8}{9.6}{\rmdefault}{\mddefault}{\updefault}{\color[rgb]{0,0,0}$D''$}%
}}}}
\put(10051,-4561){\makebox(0,0)[lb]{\smash{{\SetFigFont{8}{9.6}{\rmdefault}{\mddefault}{\updefault}{\color[rgb]{0,0,0}$D'$}%
}}}}
\put(8701,-5911){\makebox(0,0)[lb]{\smash{{\SetFigFont{8}{9.6}{\rmdefault}{\mddefault}{\updefault}{\color[rgb]{0,0,0}$d''$}%
}}}}
\put(10126,-5911){\makebox(0,0)[lb]{\smash{{\SetFigFont{8}{9.6}{\rmdefault}{\mddefault}{\updefault}{\color[rgb]{0,0,0}$d'$}%
}}}}
\put(10726,-5686){\makebox(0,0)[lb]{\smash{{\SetFigFont{6}{7.2}{\rmdefault}{\mddefault}{\updefault}{\color[rgb]{0,0,0}$2$}%
}}}}
\put(10351,-6511){\makebox(0,0)[lb]{\smash{{\SetFigFont{6}{7.2}{\rmdefault}{\mddefault}{\updefault}{\color[rgb]{0,0,0}$3$}%
}}}}
\put(9676,-5686){\makebox(0,0)[lb]{\smash{{\SetFigFont{6}{7.2}{\rmdefault}{\mddefault}{\updefault}{\color[rgb]{0,0,0}$1$}%
}}}}
\put(8701,-6511){\makebox(0,0)[lb]{\smash{{\SetFigFont{6}{7.2}{\rmdefault}{\mddefault}{\updefault}{\color[rgb]{0,0,0}$3$}%
}}}}
\put(8251,-6061){\makebox(0,0)[lb]{\smash{{\SetFigFont{6}{7.2}{\rmdefault}{\mddefault}{\updefault}{\color[rgb]{0,0,0}$2$}%
}}}}
\put(9376,-6061){\makebox(0,0)[lb]{\smash{{\SetFigFont{6}{7.2}{\rmdefault}{\mddefault}{\updefault}{\color[rgb]{0,0,0}$1$}%
}}}}
\put(9676,-4336){\makebox(0,0)[lb]{\smash{{\SetFigFont{6}{7.2}{\rmdefault}{\mddefault}{\updefault}{\color[rgb]{0,0,0}$1$}%
}}}}
\put(9376,-4711){\makebox(0,0)[lb]{\smash{{\SetFigFont{6}{7.2}{\rmdefault}{\mddefault}{\updefault}{\color[rgb]{0,0,0}$1$}%
}}}}
\put(10351,-3886){\makebox(0,0)[lb]{\smash{{\SetFigFont{6}{7.2}{\rmdefault}{\mddefault}{\updefault}{\color[rgb]{0,0,0}$3$}%
}}}}
\put(8626,-3886){\makebox(0,0)[lb]{\smash{{\SetFigFont{6}{7.2}{\rmdefault}{\mddefault}{\updefault}{\color[rgb]{0,0,0}$3$}%
}}}}
\put(8251,-4786){\makebox(0,0)[lb]{\smash{{\SetFigFont{6}{7.2}{\rmdefault}{\mddefault}{\updefault}{\color[rgb]{0,0,0}$2$}%
}}}}
\put(10726,-4786){\makebox(0,0)[lb]{\smash{{\SetFigFont{6}{7.2}{\rmdefault}{\mddefault}{\updefault}{\color[rgb]{0,0,0}$2$}%
}}}}
\put(3376,-2086){\makebox(0,0)[lb]{\smash{{\SetFigFont{6}{7.2}{\rmdefault}{\mddefault}{\updefault}{\color[rgb]{0,0,0}$1$}%
}}}}
\put(3676,-2536){\makebox(0,0)[lb]{\smash{{\SetFigFont{6}{7.2}{\rmdefault}{\mddefault}{\updefault}{\color[rgb]{0,0,0}$1$}%
}}}}
\put(3376,-3436){\makebox(0,0)[lb]{\smash{{\SetFigFont{6}{7.2}{\rmdefault}{\mddefault}{\updefault}{\color[rgb]{0,0,0}$1$}%
}}}}
\put(3676,-3886){\makebox(0,0)[lb]{\smash{{\SetFigFont{6}{7.2}{\rmdefault}{\mddefault}{\updefault}{\color[rgb]{0,0,0}$1$}%
}}}}
\put(4726,-3361){\makebox(0,0)[lb]{\smash{{\SetFigFont{6}{7.2}{\rmdefault}{\mddefault}{\updefault}{\color[rgb]{0,0,0}$2$}%
}}}}
\put(4726,-2011){\makebox(0,0)[lb]{\smash{{\SetFigFont{6}{7.2}{\rmdefault}{\mddefault}{\updefault}{\color[rgb]{0,0,0}$2$}%
}}}}
\put(4276,-1636){\makebox(0,0)[lb]{\smash{{\SetFigFont{6}{7.2}{\rmdefault}{\mddefault}{\updefault}{\color[rgb]{0,0,0}$3$}%
}}}}
\put(2701,-4261){\makebox(0,0)[lb]{\smash{{\SetFigFont{6}{7.2}{\rmdefault}{\mddefault}{\updefault}{\color[rgb]{0,0,0}$3$}%
}}}}
\put(4276,-4261){\makebox(0,0)[lb]{\smash{{\SetFigFont{6}{7.2}{\rmdefault}{\mddefault}{\updefault}{\color[rgb]{0,0,0}$3$}%
}}}}
\put(2251,-3361){\makebox(0,0)[lb]{\smash{{\SetFigFont{6}{7.2}{\rmdefault}{\mddefault}{\updefault}{\color[rgb]{0,0,0}$2$}%
}}}}
\put(2251,-2536){\makebox(0,0)[lb]{\smash{{\SetFigFont{6}{7.2}{\rmdefault}{\mddefault}{\updefault}{\color[rgb]{0,0,0}$2$}%
}}}}
\put(2701,-1636){\makebox(0,0)[lb]{\smash{{\SetFigFont{6}{7.2}{\rmdefault}{\mddefault}{\updefault}{\color[rgb]{0,0,0}$3$}%
}}}}
\end{picture}%

%% file: handle_decomp_5.pstex_t
\begin{picture}(0,0)%
\includegraphics{handle_decomp_5.pstex}%
\end{picture}%
\setlength{\unitlength}{1184sp}%
\begingroup\makeatletter\ifx\SetFigFont\undefined%
\gdef\SetFigFont#1#2#3#4#5{%
  \reset@font\fontsize{#1}{#2pt}%
  \fontfamily{#3}\fontseries{#4}\fontshape{#5}%
  \selectfont}%
\fi\endgroup%
\begin{picture}(11116,9024)(1343,-8473)
\put(8326,-3661){\makebox(0,0)[lb]{\smash{{\SetFigFont{6}{7.2}{\rmdefault}{\mddefault}{\updefault}{\color[rgb]{0,0,0}$5$}%
}}}}
\put(8326,-4486){\makebox(0,0)[lb]{\smash{{\SetFigFont{6}{7.2}{\rmdefault}{\mddefault}{\updefault}{\color[rgb]{0,0,0}$4$}%
}}}}
\put(3526,-7336){\makebox(0,0)[lb]{\smash{{\SetFigFont{6}{7.2}{\rmdefault}{\mddefault}{\updefault}{\color[rgb]{0,0,0}$4$}%
}}}}
\put(4651,-4111){\makebox(0,0)[lb]{\smash{{\SetFigFont{8}{9.6}{\rmdefault}{\mddefault}{\updefault}{\color[rgb]{0,0,0}$a$}%
}}}}
\put(8776,-4111){\makebox(0,0)[lb]{\smash{{\SetFigFont{8}{9.6}{\rmdefault}{\mddefault}{\updefault}{\color[rgb]{0,0,0}$A$}%
}}}}
\put(5776,-2311){\makebox(0,0)[lb]{\smash{{\SetFigFont{8}{9.6}{\rmdefault}{\mddefault}{\updefault}{\color[rgb]{0,0,0}$c'$}%
}}}}
\put(11776,-4111){\makebox(0,0)[lb]{\smash{{\SetFigFont{8}{9.6}{\rmdefault}{\mddefault}{\updefault}{\color[rgb]{0,0,0}$d''$}%
}}}}
\put(7501,-5911){\makebox(0,0)[lb]{\smash{{\SetFigFont{8}{9.6}{\rmdefault}{\mddefault}{\updefault}{\color[rgb]{0,0,0}$C''$}%
}}}}
\put(5701,-5911){\makebox(0,0)[lb]{\smash{{\SetFigFont{8}{9.6}{\rmdefault}{\mddefault}{\updefault}{\color[rgb]{0,0,0}$C'$}%
}}}}
\put(3901,-511){\makebox(0,0)[lb]{\smash{{\SetFigFont{8}{9.6}{\rmdefault}{\mddefault}{\updefault}{\color[rgb]{0,0,0}$D''$}%
}}}}
\put(1501,-4111){\makebox(0,0)[lb]{\smash{{\SetFigFont{8}{9.6}{\rmdefault}{\mddefault}{\updefault}{\color[rgb]{0,0,0}$D'$}%
}}}}
\put(5251,-3661){\makebox(0,0)[lb]{\smash{{\SetFigFont{6}{7.2}{\rmdefault}{\mddefault}{\updefault}{\color[rgb]{0,0,0}$5$}%
}}}}
\put(5251,-4486){\makebox(0,0)[lb]{\smash{{\SetFigFont{6}{7.2}{\rmdefault}{\mddefault}{\updefault}{\color[rgb]{0,0,0}$4$}%
}}}}
\put(6451,-2461){\makebox(0,0)[lb]{\smash{{\SetFigFont{6}{7.2}{\rmdefault}{\mddefault}{\updefault}{\color[rgb]{0,0,0}$1$}%
}}}}
\put(7126,-2461){\makebox(0,0)[lb]{\smash{{\SetFigFont{6}{7.2}{\rmdefault}{\mddefault}{\updefault}{\color[rgb]{0,0,0}$1$}%
}}}}
\put(6526,-5686){\makebox(0,0)[lb]{\smash{{\SetFigFont{6}{7.2}{\rmdefault}{\mddefault}{\updefault}{\color[rgb]{0,0,0}$1$}%
}}}}
\put(5701,-6511){\makebox(0,0)[lb]{\smash{{\SetFigFont{6}{7.2}{\rmdefault}{\mddefault}{\updefault}{\color[rgb]{0,0,0}$3$}%
}}}}
\put(7651,-6511){\makebox(0,0)[lb]{\smash{{\SetFigFont{6}{7.2}{\rmdefault}{\mddefault}{\updefault}{\color[rgb]{0,0,0}$3$}%
}}}}
\put(8251,-5461){\makebox(0,0)[lb]{\smash{{\SetFigFont{6}{7.2}{\rmdefault}{\mddefault}{\updefault}{\color[rgb]{0,0,0}$2$}%
}}}}
\put(9526,-3886){\makebox(0,0)[lb]{\smash{{\SetFigFont{6}{7.2}{\rmdefault}{\mddefault}{\updefault}{\color[rgb]{0,0,0}$1$}%
}}}}
\put(8251,-2611){\makebox(0,0)[lb]{\smash{{\SetFigFont{6}{7.2}{\rmdefault}{\mddefault}{\updefault}{\color[rgb]{0,0,0}$2$}%
}}}}
\put(7726,-1636){\makebox(0,0)[lb]{\smash{{\SetFigFont{6}{7.2}{\rmdefault}{\mddefault}{\updefault}{\color[rgb]{0,0,0}$3$}%
}}}}
\put(5701,-1636){\makebox(0,0)[lb]{\smash{{\SetFigFont{6}{7.2}{\rmdefault}{\mddefault}{\updefault}{\color[rgb]{0,0,0}$3$}%
}}}}
\put(4126,-3886){\makebox(0,0)[lb]{\smash{{\SetFigFont{6}{7.2}{\rmdefault}{\mddefault}{\updefault}{\color[rgb]{0,0,0}$1$}%
}}}}
\put(7126,-5686){\makebox(0,0)[lb]{\smash{{\SetFigFont{6}{7.2}{\rmdefault}{\mddefault}{\updefault}{\color[rgb]{0,0,0}$1$}%
}}}}
\put(5326,-5536){\makebox(0,0)[lb]{\smash{{\SetFigFont{6}{7.2}{\rmdefault}{\mddefault}{\updefault}{\color[rgb]{0,0,0}$2$}%
}}}}
\put(5326,-2611){\makebox(0,0)[lb]{\smash{{\SetFigFont{6}{7.2}{\rmdefault}{\mddefault}{\updefault}{\color[rgb]{0,0,0}$2$}%
}}}}
\put(9376,-7711){\makebox(0,0)[lb]{\smash{{\SetFigFont{8}{9.6}{\rmdefault}{\mddefault}{\updefault}{\color[rgb]{0,0,0}$d'$}%
}}}}
\put(9376,-511){\makebox(0,0)[lb]{\smash{{\SetFigFont{8}{9.6}{\rmdefault}{\mddefault}{\updefault}{\color[rgb]{0,0,0}$B$}%
}}}}
\put(8926,-211){\makebox(0,0)[lb]{\smash{{\SetFigFont{6}{7.2}{\rmdefault}{\mddefault}{\updefault}{\color[rgb]{0,0,0}$4$}%
}}}}
\put(10051,-811){\makebox(0,0)[lb]{\smash{{\SetFigFont{6}{7.2}{\rmdefault}{\mddefault}{\updefault}{\color[rgb]{0,0,0}$5$}%
}}}}
\put(9301,-1111){\makebox(0,0)[lb]{\smash{{\SetFigFont{6}{7.2}{\rmdefault}{\mddefault}{\updefault}{\color[rgb]{0,0,0}$1$}%
}}}}
\put(11626,-3361){\makebox(0,0)[lb]{\smash{{\SetFigFont{6}{7.2}{\rmdefault}{\mddefault}{\updefault}{\color[rgb]{0,0,0}$2$}%
}}}}
\put(11701,-4711){\makebox(0,0)[lb]{\smash{{\SetFigFont{6}{7.2}{\rmdefault}{\mddefault}{\updefault}{\color[rgb]{0,0,0}$1$}%
}}}}
\put(10051,-7261){\makebox(0,0)[lb]{\smash{{\SetFigFont{6}{7.2}{\rmdefault}{\mddefault}{\updefault}{\color[rgb]{0,0,0}$1$}%
}}}}
\put(9226,-7036){\makebox(0,0)[lb]{\smash{{\SetFigFont{6}{7.2}{\rmdefault}{\mddefault}{\updefault}{\color[rgb]{0,0,0}$3$}%
}}}}
\put(8851,-7936){\makebox(0,0)[lb]{\smash{{\SetFigFont{6}{7.2}{\rmdefault}{\mddefault}{\updefault}{\color[rgb]{0,0,0}$2$}%
}}}}
\put(4726,-7861){\makebox(0,0)[lb]{\smash{{\SetFigFont{6}{7.2}{\rmdefault}{\mddefault}{\updefault}{\color[rgb]{0,0,0}$5$}%
}}}}
\put(4351,-7036){\makebox(0,0)[lb]{\smash{{\SetFigFont{6}{7.2}{\rmdefault}{\mddefault}{\updefault}{\color[rgb]{0,0,0}$1$}%
}}}}
\put(1801,-3361){\makebox(0,0)[lb]{\smash{{\SetFigFont{6}{7.2}{\rmdefault}{\mddefault}{\updefault}{\color[rgb]{0,0,0}$1$}%
}}}}
\put(2326,-3886){\makebox(0,0)[lb]{\smash{{\SetFigFont{6}{7.2}{\rmdefault}{\mddefault}{\updefault}{\color[rgb]{0,0,0}$3$}%
}}}}
\put(1801,-4711){\makebox(0,0)[lb]{\smash{{\SetFigFont{6}{7.2}{\rmdefault}{\mddefault}{\updefault}{\color[rgb]{0,0,0}$2$}%
}}}}
\put(4726,-286){\makebox(0,0)[lb]{\smash{{\SetFigFont{6}{7.2}{\rmdefault}{\mddefault}{\updefault}{\color[rgb]{0,0,0}$2$}%
}}}}
\put(4276,-1111){\makebox(0,0)[lb]{\smash{{\SetFigFont{6}{7.2}{\rmdefault}{\mddefault}{\updefault}{\color[rgb]{0,0,0}$3$}%
}}}}
\put(3451,-811){\makebox(0,0)[lb]{\smash{{\SetFigFont{6}{7.2}{\rmdefault}{\mddefault}{\updefault}{\color[rgb]{0,0,0}$1$}%
}}}}
\put(11251,-3886){\makebox(0,0)[lb]{\smash{{\SetFigFont{6}{7.2}{\rmdefault}{\mddefault}{\updefault}{\color[rgb]{0,0,0}$3$}%
}}}}
\put(7576,-2311){\makebox(0,0)[lb]{\smash{{\SetFigFont{8}{9.6}{\rmdefault}{\mddefault}{\updefault}{\color[rgb]{0,0,0}$c''$}%
}}}}
\put(4051,-7711){\makebox(0,0)[lb]{\smash{{\SetFigFont{8}{9.6}{\rmdefault}{\mddefault}{\updefault}{\color[rgb]{0,0,0}$b$}%
}}}}
\end{picture}%

%% file: polyhedron.pstex_t
\begin{picture}(0,0)%
\includegraphics{polyhedron.pstex}%
\end{picture}%
\setlength{\unitlength}{1184sp}%
\begingroup\makeatletter\ifx\SetFigFont\undefined%
\gdef\SetFigFont#1#2#3#4#5{%
  \reset@font\fontsize{#1}{#2pt}%
  \fontfamily{#3}\fontseries{#4}\fontshape{#5}%
  \selectfont}%
\fi\endgroup%
\begin{picture}(7227,7713)(3586,-8173)
\put(4126,-2236){\makebox(0,0)[lb]{\smash{{\SetFigFont{9}{10.8}{\rmdefault}{\mddefault}{\updefault}{\color[rgb]{0,0,0}$D'$}%
}}}}
\put(5776,-811){\makebox(0,0)[lb]{\smash{{\SetFigFont{9}{10.8}{\rmdefault}{\mddefault}{\updefault}{\color[rgb]{0,0,0}$D''$}%
}}}}
\put(9526,-2236){\makebox(0,0)[lb]{\smash{{\SetFigFont{9}{10.8}{\rmdefault}{\mddefault}{\updefault}{\color[rgb]{0,0,0}$B$}%
}}}}
\put(3601,-5986){\makebox(0,0)[lb]{\smash{{\SetFigFont{9}{10.8}{\rmdefault}{\mddefault}{\updefault}{\color[rgb]{0,0,0}$a$}%
}}}}
\put(10276,-6061){\makebox(0,0)[lb]{\smash{{\SetFigFont{9}{10.8}{\rmdefault}{\mddefault}{\updefault}{\color[rgb]{0,0,0}$c''$}%
}}}}
\put(5026,-5611){\makebox(0,0)[lb]{\smash{{\SetFigFont{9}{10.8}{\rmdefault}{\mddefault}{\updefault}{\color[rgb]{0,0,0}$C'$}%
}}}}
\put(8701,-5611){\makebox(0,0)[lb]{\smash{{\SetFigFont{9}{10.8}{\rmdefault}{\mddefault}{\updefault}{\color[rgb]{0,0,0}$A$}%
}}}}
\put(5026,-4186){\makebox(0,0)[lb]{\smash{{\SetFigFont{9}{10.8}{\rmdefault}{\mddefault}{\updefault}{\color[rgb]{0,0,0}$b$}%
}}}}
\put(8701,-4186){\makebox(0,0)[lb]{\smash{{\SetFigFont{9}{10.8}{\rmdefault}{\mddefault}{\updefault}{\color[rgb]{0,0,0}$d''$}%
}}}}
\put(8326,-8011){\makebox(0,0)[lb]{\smash{{\SetFigFont{9}{10.8}{\rmdefault}{\mddefault}{\updefault}{\color[rgb]{0,0,0}$c'$}%
}}}}
\put(7276,-4561){\makebox(0,0)[lb]{\smash{{\SetFigFont{9}{10.8}{\rmdefault}{\mddefault}{\updefault}{\color[rgb]{0,0,0}$d'$}%
}}}}
\put(6826,-6136){\makebox(0,0)[lb]{\smash{{\SetFigFont{9}{10.8}{\rmdefault}{\mddefault}{\updefault}{\color[rgb]{0,0,0}$C''$}%
}}}}
\end{picture}%

%% file: chain_link_with_gen.pstex_t
\begin{picture}(0,0)%
\includegraphics{chain_link_with_gen.pstex}%
\end{picture}%
\setlength{\unitlength}{1184sp}%
\begingroup\makeatletter\ifx\SetFigFont\undefined%
\gdef\SetFigFont#1#2#3#4#5{%
  \reset@font\fontsize{#1}{#2pt}%
  \fontfamily{#3}\fontseries{#4}\fontshape{#5}%
  \selectfont}%
\fi\endgroup%
\begin{picture}(9518,7983)(-463,-8845)
\put(6901,-4111){\makebox(0,0)[lb]{\smash{{\SetFigFont{9}{10.8}{\rmdefault}{\mddefault}{\updefault}{\color[rgb]{0,0,0}$Z$}%
}}}}
\put(4426,-7111){\makebox(0,0)[lb]{\smash{{\SetFigFont{9}{10.8}{\rmdefault}{\mddefault}{\updefault}{\color[rgb]{0,0,0}$X$}%
}}}}
\put(2026,-4261){\makebox(0,0)[lb]{\smash{{\SetFigFont{9}{10.8}{\rmdefault}{\mddefault}{\updefault}{\color[rgb]{0,0,0}$Y$}%
}}}}
\put(6151,-1261){\makebox(0,0)[lb]{\smash{{\SetFigFont{9}{10.8}{\rmdefault}{\mddefault}{\updefault}{\color[rgb]{0,0,0}$p_z$}%
}}}}
\put(901,-3136){\makebox(0,0)[lb]{\smash{{\SetFigFont{9}{10.8}{\rmdefault}{\mddefault}{\updefault}{\color[rgb]{0,0,0}$p_y$}%
}}}}
\put(3001,-8686){\makebox(0,0)[lb]{\smash{{\SetFigFont{9}{10.8}{\rmdefault}{\mddefault}{\updefault}{\color[rgb]{0,0,0}$p_x$}%
}}}}
\end{picture}%

%% file: sliced_link_complement.pstex_t
\begin{picture}(0,0)%
\includegraphics{sliced_link_complement.pstex}%
\end{picture}%
\setlength{\unitlength}{1184sp}%
\begingroup\makeatletter\ifx\SetFigFont\undefined%
\gdef\SetFigFont#1#2#3#4#5{%
  \reset@font\fontsize{#1}{#2pt}%
  \fontfamily{#3}\fontseries{#4}\fontshape{#5}%
  \selectfont}%
\fi\endgroup%
\begin{picture}(9624,4824)(1189,-6373)
\put(1576,-3961){\makebox(0,0)[lb]{\smash{{\SetFigFont{8}{9.6}{\rmdefault}{\mddefault}{\updefault}{\color[rgb]{0,0,0}$y$}%
}}}}
\put(5701,-5611){\makebox(0,0)[lb]{\smash{{\SetFigFont{8}{9.6}{\rmdefault}{\mddefault}{\updefault}{\color[rgb]{0,0,0}$X$}%
}}}}
\put(3226,-3661){\makebox(0,0)[lb]{\smash{{\SetFigFont{8}{9.6}{\rmdefault}{\mddefault}{\updefault}{\color[rgb]{0,0,0}$Y$}%
}}}}
\put(7801,-3361){\makebox(0,0)[lb]{\smash{{\SetFigFont{8}{9.6}{\rmdefault}{\mddefault}{\updefault}{\color[rgb]{0,0,0}$Z$}%
}}}}
\put(3601,-5836){\makebox(0,0)[lb]{\smash{{\SetFigFont{8}{9.6}{\rmdefault}{\mddefault}{\updefault}{\color[rgb]{0,0,0}$x$}%
}}}}
\put(9751,-3961){\makebox(0,0)[lb]{\smash{{\SetFigFont{8}{9.6}{\rmdefault}{\mddefault}{\updefault}{\color[rgb]{0,0,0}$z$}%
}}}}
\end{picture}%

%% file: two_drums.pstex_t
\begin{picture}(0,0)%
\includegraphics{two_drums.pstex}%
\end{picture}%
\setlength{\unitlength}{1184sp}%
\begingroup\makeatletter\ifx\SetFigFont\undefined%
\gdef\SetFigFont#1#2#3#4#5{%
  \reset@font\fontsize{#1}{#2pt}%
  \fontfamily{#3}\fontseries{#4}\fontshape{#5}%
  \selectfont}%
\fi\endgroup%
\begin{picture}(5427,6774)(3439,-6973)
\put(8851,-3961){\makebox(0,0)[lb]{\smash{{\SetFigFont{9}{10.8}{\rmdefault}{\mddefault}{\updefault}{\color[rgb]{0,0,0}$Z$}%
}}}}
\put(8851,-1561){\makebox(0,0)[lb]{\smash{{\SetFigFont{9}{10.8}{\rmdefault}{\mddefault}{\updefault}{\color[rgb]{0,0,0}$y$}%
}}}}
\put(4426,-4711){\makebox(0,0)[lb]{\smash{{\SetFigFont{9}{10.8}{\rmdefault}{\mddefault}{\updefault}{\color[rgb]{0,0,0}$Y$}%
}}}}
\put(6826,-4786){\makebox(0,0)[lb]{\smash{{\SetFigFont{9}{10.8}{\rmdefault}{\mddefault}{\updefault}{\color[rgb]{0,0,0}$X$}%
}}}}
\put(4501,-2461){\makebox(0,0)[lb]{\smash{{\SetFigFont{9}{10.8}{\rmdefault}{\mddefault}{\updefault}{\color[rgb]{0,0,0}$x$}%
}}}}
\put(6901,-2461){\makebox(0,0)[lb]{\smash{{\SetFigFont{9}{10.8}{\rmdefault}{\mddefault}{\updefault}{\color[rgb]{0,0,0}$z$}%
}}}}
\end{picture}%

%% file: splitted_drums.pstex_t
\begin{picture}(0,0)%
\includegraphics{splitted_drums.pstex}%
\end{picture}%
\setlength{\unitlength}{1184sp}%
\begingroup\makeatletter\ifx\SetFigFont\undefined%
\gdef\SetFigFont#1#2#3#4#5{%
  \reset@font\fontsize{#1}{#2pt}%
  \fontfamily{#3}\fontseries{#4}\fontshape{#5}%
  \selectfont}%
\fi\endgroup%
\begin{picture}(6855,9288)(2536,-8023)
\put(6301,914){\makebox(0,0)[rb]{\smash{{\SetFigFont{9}{10.8}{\rmdefault}{\mddefault}{\updefault}{\color[rgb]{0,0,0}$6$}%
}}}}
\put(9376,-1111){\makebox(0,0)[rb]{\smash{{\SetFigFont{9}{10.8}{\rmdefault}{\mddefault}{\updefault}{\color[rgb]{0,0,0}$5$}%
}}}}
\put(9076,-4636){\makebox(0,0)[rb]{\smash{{\SetFigFont{9}{10.8}{\rmdefault}{\mddefault}{\updefault}{\color[rgb]{0,0,0}$3$}%
}}}}
\put(4651,-7486){\makebox(0,0)[rb]{\smash{{\SetFigFont{9}{10.8}{\rmdefault}{\mddefault}{\updefault}{\color[rgb]{0,0,0}$1$}%
}}}}
\put(3376,-1786){\makebox(0,0)[rb]{\smash{{\SetFigFont{9}{10.8}{\rmdefault}{\mddefault}{\updefault}{\color[rgb]{0,0,0}$4$}%
}}}}
\put(2551,-4936){\makebox(0,0)[lb]{\smash{{\SetFigFont{9}{10.8}{\rmdefault}{\mddefault}{\updefault}{\color[rgb]{0,0,0}$2$}%
}}}}
\end{picture}%

%% file: chain_polyhedron.pstex_t
\begin{picture}(0,0)%
\includegraphics{chain_polyhedron.pstex}%
\end{picture}%
\setlength{\unitlength}{1184sp}%
\begingroup\makeatletter\ifx\SetFigFont\undefined%
\gdef\SetFigFont#1#2#3#4#5{%
  \reset@font\fontsize{#1}{#2pt}%
  \fontfamily{#3}\fontseries{#4}\fontshape{#5}%
  \selectfont}%
\fi\endgroup%
\begin{picture}(7602,8067)(3211,-8377)
\put(3226,-5761){\makebox(0,0)[lb]{\smash{{\SetFigFont{9}{10.8}{\rmdefault}{\mddefault}{\updefault}{\color[rgb]{0,0,0}$6$}%
}}}}
\put(8176,-8236){\makebox(0,0)[lb]{\smash{{\SetFigFont{9}{10.8}{\rmdefault}{\mddefault}{\updefault}{\color[rgb]{0,0,0}$2$}%
}}}}
\put(10426,-5836){\makebox(0,0)[lb]{\smash{{\SetFigFont{9}{10.8}{\rmdefault}{\mddefault}{\updefault}{\color[rgb]{0,0,0}$4$}%
}}}}
\put(9826,-2161){\makebox(0,0)[lb]{\smash{{\SetFigFont{9}{10.8}{\rmdefault}{\mddefault}{\updefault}{\color[rgb]{0,0,0}$3$}%
}}}}
\put(6226,-661){\makebox(0,0)[lb]{\smash{{\SetFigFont{9}{10.8}{\rmdefault}{\mddefault}{\updefault}{\color[rgb]{0,0,0}$5$}%
}}}}
\put(3676,-2386){\makebox(0,0)[lb]{\smash{{\SetFigFont{9}{10.8}{\rmdefault}{\mddefault}{\updefault}{\color[rgb]{0,0,0}$1$}%
}}}}
\end{picture}%